\documentclass[11pt,a4paper]{article}
\usepackage[english]{babel}
\usepackage{fancyhdr}
\usepackage[latin1]{inputenc}
\usepackage{amsmath}
\usepackage{amsfonts}
\usepackage{amssymb}
\usepackage{graphicx}
\usepackage{latexsym}
\usepackage{amsthm}
\usepackage{color}
\usepackage{midpage}
\usepackage[width=15.00cm]{geometry}
\usepackage{amscd}
\usepackage{tikz-cd}
\usepackage{hyperref}

\numberwithin{equation}{section}
\theoremstyle{plain}
\newtheorem{thm}{Theorem}[section]
\theoremstyle{remark}
\newtheorem{rem}{Remark}[section]
\theoremstyle{cor}
\newtheorem{cor}[thm]{Corollary}
\theoremstyle{lem}
\newtheorem{lem}[thm]{Lemma}
\theoremstyle{prop}

\DeclareMathOperator{\Var}{Var}
\providecommand{\keywords}[1]{\textbf{Keywords and Phrases:} #1}
\providecommand{\classification}[1]{\textbf{AMS Classification:}#1}

\begin{document}

\author{{Anna Paola Todino} \\
{\small Ruhr-Universit\"at Bochum, Germany}\\
{\small anna.todino@ruhr-uni-bochum.de}}
\date{}
\title{Limiting Behavior for the Excursion Area of Band-Limited Spherical
Random Fields}
\maketitle

\begin{abstract}
In this paper, we investigate some geometric functionals for band-limited
Gaussian and isotropic spherical random fields in dimension 2. In
particular, we focus on the area of excursion sets, providing its behavior
in the high energy limit. Our result is based on Wiener chaos expansion for
non linear transform of Gaussian fields and on an explicit derivation on the
high-frequency limit of the covariance function of the field. As a simple corollary we establish also the Central Limit Theorem for the excursion area.

\end{abstract}


\begin{itemize} 
\item \keywords{} 
Gaussian Eigenfunctions, Excursion Area, Wiener-chaos expansion, Hilb's asymptotics, Central Limit Theorem.

\item \classification{} 60G60, 42C10, 33C55, 60F05.
\end{itemize}

\section{Introduction and Background}
Let $\{T_{\ell}(x),x\in \mathbb{S}^{2}\}$ denote the spherical harmonics, which are solutions of the Helmholtz equation: 
\begin{equation*}
\Delta _{\mathbb{S}^{2}}T_{\ell  }(x)+\ell (\ell +1)T_{\ell  }(x)=0,\mbox{ }\ell
=1,2,\dots;
\end{equation*}
where $\Delta _{\mathbb{S}^{2}}$ is the spherical Laplacian. We can put on these eigenfunctions a random structure such that $\{T_{\ell }(x),x\in \mathbb{S}^{2}\}$ are isotropic, centred Gaussian, with covariance function given by
\begin{equation*}
\mathbb{E}[T_{\ell }(x)T_{\ell }(y)]=\frac{2\ell +1}{4\pi }P_{\ell }(\cos
d(x,y)),
\end{equation*}
where $P_{\ell }$ is the Legendre polynomial and $d(x,y)$ the spherical
geodesic distance between $x$ and $y$, $d(x,y)=\arccos (\langle x,y\rangle )$.
After choosing a standard basis $\{Y_{\ell m}(x)\}$ of $L^2(\mathbb{S}^2)$, the random fields $\{T_{\ell }(x)\}$ can be expressed by
\begin{equation*}
T_{\ell }(x)=\sum_{m=-\ell }^{\ell }a_{\ell m}Y_{\ell m}(x),
\end{equation*}
where  $\{a_{\ell m}\}$ is the array of random spherical harmonic coefficients, which are independent, safe for the condition $\bar{a}_{\ell m} = (-1)^m a_{\ell,-m};$ for $m \ne0$ they are standard complex-valued Gaussian variables, while $a_{\ell 0}$ is a standard real-valued Gaussian variable. They satisfy
\begin{equation*}
\mathbb{E}[a_{\ell m}\bar{a}_{\ell ^{\prime }m^{\prime }}]=\delta _{\ell
}^{\ell ^{\prime }}\delta _{m}^{m^{\prime }}.
\end{equation*}
The geometry of the excursion sets of random eigenfunctions has been studied in many different papers, among them \cite{M e W 2012}, \cite{M e W 2011bis}, \cite{M e W 2011}, \cite{MRW2017}, \cite{W}, \cite{CMW2016F}, \cite{CM18}, \cite{35} and \cite{Todino1}, \cite{Todino2} for subdomains of $\mathbb{S}^2$ (see also \cite{M e Mau 2015} and \cite{Rossi} for the d-dimensional sphere $\mathbb{S}^d$). In this framework we aim to extend these results to the case of 
band-limited functions (see also \cite{BW},
\cite{NS09}, \cite{NS17}, \cite{SW19}, \cite{SW14}).\\

Hence let us consider the sequence $\alpha_{n,\beta}$ given by
\begin{equation}\label{alpha}
\alpha_{n,\beta}:= \sqrt{1-\frac{1}{n^\beta}}
\end{equation}
with $0 < \beta < 1, \beta \in \mathbb{R}.$ The band-limited functions here are random fields $\{\bar{T}_{\alpha_{n,\beta} }(x),x\in \mathbb{S}^{2}\}$ defined as 
\begin{equation}
\bar{T}_{\alpha_{n,\beta} }(x)=\sqrt{C_{n,\beta }}\sum_{\ell =\alpha_{n,\beta} n}^{n}T_{\ell
}(x),  \label{1}
\end{equation}%
where 
\begin{equation*}
C_{n,\beta }:=\frac{4\pi }{n^{2}(1-\alpha_{n,\beta} ^{2})+2n+1}= \frac{4\pi}{n^{2-\beta}+2n+1}.
\end{equation*}%
$\{\bar{T}_{\alpha_{n,\beta} }(x)\}$ are centred Gaussian with $\mathbb{E}[\bar{T}_{\alpha_{n,\beta} }(x)^{2}]=1$ and covariance
function given by 
\begin{equation}
\bar{\Gamma}_{\alpha_{n,\beta} }(x,y)= C_{n,\beta } \left(\sum_{\ell=\alpha_{n,\beta} n}^n \mathbb{E}[T_\ell(x)T_\ell(y)]\right)=C_{n,\beta }\sum_{\ell =\alpha_{n,\beta} n}^{n}\frac{%
	2\ell +1}{4\pi }P_{\ell }(\cos d(x,y)).  \label{cov}
\end{equation}%
We consider the excursion sets, defined as: 
\begin{equation}
A_{u}(\bar{T}_{\alpha_{n,\beta} }):=\{x\in \mathbb{S}^{2}:\bar{T}_{\alpha_{n,\beta} }(x)\geq u\},
\label{exsets}
\end{equation}%
with $u\in \mathbb{R}$, $u \ne 0$; in this paper we focus on the area of these regions, which we denote by $S_{\alpha_{n,\beta}}(u)$. Along the lines of \cite{M e W 2011}, it can be written as a functional of the random field itself in the following way
\begin{equation*}
S_{\alpha_{n,\beta}}(u)=\int_{\mathbb{S}^2} 1_{\{\bar{T}_{\alpha_{n,\beta}} (x) >u\}} \,dx,
\end{equation*}
where $1(\cdot)$ is, as usual, the characteristic function which takes value
one if the condition in the argument is satisfied, zero otherwise.  This
expression allows to project the area into the orthonormal system generated by
Hermite polynomials (Wiener chaoses projection) $H_{k}(u)$, $k \in \mathbb{N}$, that is
\begin{equation}
H_{0}(u)=1, H_{1}(u)=u, H_{2}(u)=u^{2}-1, \dots, H_{k}(t) = t H_{k-1}(t) - H^\prime_{k-1} (t ), k \geq 1.
\end{equation}
Indeed, since $1(\cdot) \in L^2(\mathbb{S}^2)$, it can
be expanded as
\begin{equation*}
1_{\{\bar{T}_{\alpha_{n,\beta}} (x) >u\}}=\sum_{q=0}^{\infty} \frac{J_q(u)}{q!} H_q(\bar{%
	T}_{\alpha_{n,\beta}}(x) ),
\end{equation*}
in $L^2(\Omega).$
The coefficients $\{J_q(\cdot)\}$ have the analytic expressions $%
J_0(u)=\Phi(u),$ $J_1(u)=-\phi(u),$ $J_2(u)=-u\phi(u),$ $J_3(u)=(1-u^2)%
\phi(u)$ and in general 
\begin{equation*}
J_q(u)=-H_{q-1}(u)\phi(u),
\end{equation*}
where $\phi(\cdot)$ and $\Phi(\cdot)$ are the density function and the
distribution function of a standard Gaussian variable (\cite{M e W 2011} and 
\cite{Peccati Nourdin}).
It follows that 
\begin{equation}\label{expansion}
S_{\alpha_{n,\beta}}(u)= \sum_{q=0}^{\infty} \dfrac{J_q(u)}{q!} \int_{\mathbb{S}^2} H_q(%
\bar{T}_{\alpha_{n,\beta}} (x) ) \, dx.
\end{equation}
Note that if we consider $\beta \equiv 1$, the random field in (\ref{1}) is the eigenfunction $T_\ell$ and it was shown in \cite{M e W 2011} that, in this case, 
the projection on the first component vanishes identically and the whole series in (\ref{expansion}) is dominated simply by the second chaotic component. More explicitly,
the variance of this single term is asymptotically equivalent to the variance of the full series, and its asymptotic distribution (Gaussian) gives also the limiting behavior of the excursion area. 
On the contrary, when $\beta \equiv 0$, the expansion in (\ref{expansion}) does not have any leading component, namely, each chaotic component has the same asymptotic behavior (as it happens
for the defect case, defined as the difference between positive and negative
regions, when only one eigenfunction is considered (see \cite{M e W
	2011bis}, \cite{Rossi})). 
It could
then be suspected that the limiting behavior may depend on the value of $%
\beta $, but this turns out not to be the case. Indeed, we will prove that for any $0<\beta< 1$ the second chaotic
component is still the leading term of the series expansion in (\ref{expansion}) and so no phase
transition with respect to $\beta$ arises.

\section{Main Result}
Let us consider the expansion of the excursion area given in (\ref{expansion}),
we can write
\begin{equation} \label{serie}
\begin{split}
S_{\alpha_{n,\beta}}(u) &= (1-\Phi(u) )\int_{\mathbb{S}^2}\,dx + \phi(u) \int_{\mathbb{S}%
	^2} H_1(\bar{T}_{\alpha_{n,\beta}} (x) ) \,dx\\&+ u\phi(u) \frac{1}{2} \int_{\mathbb{S}^2}
H_2(\bar{T}_{\alpha_{n,\beta}} (x) ) \,dx 
+ \sum_{q=3}^{\infty} \dfrac{J_q(u)}{q!} \int_{\mathbb{S}^2} H_q(\bar{T}
_{\alpha_{n,\beta}} (x) )\,dx,
\end{split}
\end{equation}
in the $L^2(\Omega)-$convergence sense.
Denoting 
\begin{equation*}
h_{{\beta};q}:=\int_{\mathbb{S}^2} H_q(\bar{T}_{\alpha_{n,\beta}} (x) ) \, dx \mbox{ }%
\mbox{ }\mbox{ }\mbox{ }\mbox{ }\mbox{ }\mbox{ } q=1,2,\dots,
\end{equation*}
we have that 
\begin{equation*}
S_{\alpha_{n,\beta}}(u)=\sum_{q=0}^{\infty} \dfrac{J_q(u)}{q!} h_{{\beta};q}.
\end{equation*}
\begin{rem}
	Note that $\int_{\mathbb{S}^2} H_1(\bar{T}_{\alpha_{n,\beta}} (x)) \,dx=0$, indeed 
	\begin{equation*}
	\int_{\mathbb{S}^2} H_1(\bar{T}_{\alpha_{n,\beta}} (x)) \,dx= \int_{\mathbb{S}^2} \bar{T}%
	_{\alpha_{n,\beta}} (x) \,dx= \sqrt{C_{n,{\beta}}} \sum_{\ell={\alpha_{n,\beta}} n}^{n} \int_{\mathbb{%
			S}^2} T_\ell(x) \, dx=0
	\end{equation*}
	thanks to the orthogonal property of Spherical Harmonics (see \cite{M e Peccati}, page 66).
\end{rem}
\begin{rem}
	The choice of $C_{n,{\beta}}$ in (\ref{cov}) is such that $%
	\Var[\bar{T}_{\alpha_{n,\beta}}(x)]=1$. Indeed,  
	\begin{equation*}  \label{key}
	\begin{split}
	\Var&[{\bar{T}_{\alpha_{n,\beta}}(x)}]=\!\dfrac{4\pi}{n^2(1-{\alpha^2_{n,\beta}})+2n+1}
	\!\!\sum_{\ell={\alpha_{n,\beta}} n}^{n}\!\!\![\Var{T_\ell(x)}]=\dfrac{4\pi}{n^2(1-{\alpha^2_{n,\beta}})+2n+1}\!\!\!
	\sum_{\ell={\alpha_{n,\beta}} n}^{n} \!\!\!\dfrac{2\ell+1}{4\pi} \\
	&=\dfrac{1}{n^2(1-{\alpha^2_{n,\beta}})+2n+1} \bigg[ 2\bigg(\dfrac{n(n+1)}{2} - \dfrac{%
		({\alpha_{n,\beta}} n-1)({\alpha_{n,\beta}} n)}{2}\bigg) +( n-{\alpha_{n,\beta}} n+1)\bigg] \\
	&= \dfrac{1}{n^2(1-{\alpha^2_{n,\beta}})+2n+1} ( n^2+n -{\alpha^2_{n,\beta}} n^2+{\alpha_{n,\beta}} n+n -{\alpha_{n,\beta}}
	n +1)=1.
	\end{split}%
	\end{equation*}
\end{rem}
The main result of this paper gives
the high energy behavior of the variances of $h_{\beta; q}$ for $q\geq2$.
\begin{thm}
	\label{main} 	For $0 < \beta < 1,$
	as $n \rightarrow \infty$, 
	\begin{align}
	\Var(h_{\beta; 2})&=%
	\dfrac{32\pi^2}{n^{2-\beta}}- \dfrac{64\pi^2}{n^{3-2\beta}} +o\left(\frac{1}{n^{3-2\beta}}\right) & \mbox{ for } q=2, \\ 
	\Var(h_{\beta;q})& = 
	O\left( \dfrac{1}{n^{2}}\right) & \mbox{  for  } q
	= 3 \mbox{ and } q \geq 5 \\
	\Var(h_{\beta;4})&=O\left( \dfrac{ \log n}{n^{2}}\right) & \mbox{  for  } q
	= 4.
	\end{align}
\end{thm}
\begin{rem}From Theorem \ref{main} it follows that the second chaos is the leading term for all $\beta \in (0,1)$. 
	If $\beta =1$, the random field in (\ref{1}) is simply one random eigenfunction and this case has already been investigated in \cite{M e W 2011}, where it has been proved that the variance of $h_{1;2}$ is $32\pi^2\dfrac{2}{2\ell+1}$. 
\end{rem}
For the continuity of the norm and the orthogonality of the
Hermite polynomials,
the following expansion holds in the $L^2(\Omega)$ sense: 
\begin{equation*}  \label{serie-var}
\mbox{Var}\lbrack\ \!S_{\alpha_{n,\beta}}(u) ] \!= \!0 +\! 0+\! \dfrac{u^2\phi(u)^2}{4} 
\mbox{Var}\bigg[ \int_{\mathbb{S}^2} H_2(\bar{T}_{\alpha_{n,\beta}} (x)) \,dx %
\bigg] + \sum_{q=3}^{\infty} \dfrac{J_q(u)^2}{q!^2} \mbox{Var}\bigg[ \!\int_{%
	\mathbb{S}^2} \!H_q(\bar{T}_{\alpha_{n,\beta}} (x)) \,dx\bigg].
\end{equation*}
Then, as a corollary, we get
\begin{cor} 
	\label{corollary} For $0 < \beta <1 $, as $n \rightarrow \infty,$  
	\begin{equation*}
	\Var(S_{\alpha_{n,\beta}}(u))= \dfrac{u^2\phi(u)^2}{4} \Var(h_{{\beta};2})+O \left(\frac{\log n}{n^{2}}\right)=32 \pi^2  \dfrac{u^2\phi(u)^2}{4} \dfrac{1}{n^{2-\beta}}+ R_\beta
	\end{equation*}
	where  
	\begin{equation}
	R_\beta= 
	\begin{cases}
	O\left(\dfrac{\log n}{n^2}\right) & \mbox{ if }  0 < \beta \leq 1/2, \\ 
	O\left(\dfrac{1}{n^{3-2\beta}}\right)  & \mbox{ if  }  \frac{1}{2}<\beta<1. \\ 
	\end{cases}%
	\end{equation}
\end{cor}
The key role of the proof of Theorem \ref{main} is played by the derivation
of the asymptotic behaviour for the covariance function in (\ref{cov}) and
it is given in Lemma \ref{Lemma1}.\\

Denoting by $N$ the North Pole, we fix $x=N$ and, in view of the isotropy,
we can write $\bar{\Gamma}_{\alpha_{n,\beta}}(x,y)=\bar{\Gamma}_{\alpha_{n,\beta}}(\cos
\theta)$ with $\theta \in [0,\pi)$. Changing variable $\theta= \frac{\psi}{\alpha_{n,\beta} n}$,  the following lemma gives the
asymptotic behaviour in the high frequency limit of the covariance function
for $1<\psi \leq \alpha_{n,\beta} n (\pi-\epsilon)$, for any $\epsilon>0$.
\begin{lem}
	\label{Lemma1}  Given $\bar{\Gamma}_{\alpha_{n,\beta}}(x,y)$ as in (\ref{cov}%
	), for $0 < \beta <1,
	\beta \in \mathbb{R},$ we have that, for $1<\psi < n^{\beta}$, as $n \rightarrow \infty$,  
	\begin{equation*}  \label{asymptotic cov function}
	\begin{split}
	\bar{\Gamma}_{\alpha_{n,\beta}}\left(\cos \frac{\psi}{\alpha_{n,\beta} n}\right)
	&=\frac{C_{n,\beta}}{4\pi}  \frac{(\alpha_{n,\beta} n)^2}{\sqrt{\psi}}   \frac{1}{n^\beta} \sqrt{\frac{2}{\pi}} \frac{1}{2} \bigg(\!\!-\cos \left(\psi-\frac{3\pi}{4}\right) \frac{\psi}{4n^{\beta}}- \sin  \left(\psi-\frac{3\pi}{4}\right) \\&+O\left( \frac{1}{ \psi} \right)\bigg)=:\bar{\Gamma}_1;
	\end{split}
	\end{equation*}
	while, for $\psi \in [n^{\beta},\alpha_{n,\beta} n (\pi-\epsilon)]$ it is 
	\begin{equation*}\label{cov2}
	\begin{split}
	\bar{\Gamma}_{\alpha_{n,\beta}}&\left(\cos \frac{\psi}{\alpha_{n,\beta} n}\right) =\frac{C_{n,\beta}}{4\pi} 2
	\sqrt{\frac{2}{\pi}}\left( -2 \sin \left(\frac{\psi}{2}-\frac{3\pi}{4}+\frac{(n+1)}{2\alpha_{n,\beta} n} \psi\right) \sin\left(\frac{n+1}{2n\alpha_{n,\beta}}\psi-\frac{\psi}{2}\right)\right) \\&\times\left(\frac{(\alpha_{n,\beta} n )^2}{\psi^{3/2}} +  \sqrt{\psi} \right)
	+O\left(C_{n,\beta}(\alpha_{n,\beta}n)^2 \frac{1}{\psi^{5/2}}\right) =:\bar{\Gamma}_2.
	\end{split}
	\end{equation*}
\end{lem}
Moreover, we observe that 
\begin{equation}\label{CLT}
\begin{split}
\int_{\mathbb{S}^2} H_2(\bar{T}_{\alpha_{n,\beta}}(x)) \, dx &= \int_{\mathbb{S}^2} (\bar{T}_{\alpha_{n,\beta}}(x)^2-1) \, dx= C_{n,\beta} \sum_{\ell}\sum_{\ell'} \int_{\mathbb{S}^2} T_\ell(x) T_{\ell'}(x) \, dx -4\pi \\&= C_{n,\beta} \sum_{\ell=\alpha_{n,\beta} n}^{n}\sum_{m=-\ell}^{\ell} |a_{\ell m}|^2 -4\pi ,
\end{split}
\end{equation}
which are sums of independent Gaussian random variables; the mean of (\ref{CLT}) is zero and then, from Theorem \ref{main}, the Central Limit Theorem follows for $h_{\beta;2}$.
As a consequence, along the same lines as in \cite{M e Mau 2015}, we can establish the validity of the Central Limit Theorem for $S_{\alpha_{n,\beta}}$. Note that from (\ref{serie}) it is easy to see that $\mathbb{E}[S_{\alpha_{n,\beta}} (u)]= (1-\Phi(u))4\pi$.
\begin{cor}\label{TLC}
	For all $0 <\beta < 1$, as $n \rightarrow \infty$, $$\dfrac{S_{\alpha_{n,\beta}} (u)-\mathbb{E}[S_{\alpha_{n,\beta}} (u)]}{\sqrt{\Var[S_{\alpha_{n,\beta}} (u)]}} \rightarrow_d Z,$$
	where $Z \sim N(0,1)$ and $\rightarrow_d $ denote convergence in distribution.
\end{cor}

\begin{rem}
	When $\beta=0$ all the chaotic components have the same asymptotic behavior. Indeed, the idea is that the covariance function in this case behaves like
	\begin{equation}
	\bar{\Gamma}_{\alpha_{n,0}}\left(\cos \frac{\psi}{n}\right) \sim C_{n,0} \frac{1}{\sin (\psi/n)} \left(\frac{\psi/n}{\sin (\psi/n)}\right)^{1/2} (n+1) \sqrt{\frac{2}{(n+1)(\psi/n)}} \cos \left(\psi- \frac{3\pi}{4}\right)
	\end{equation} so that for all $q$  $$\Var(h_{0,q}) \sim \frac{C_q}{n^2} \int_{1}^{\pi n} \frac{1}{\psi^{\frac{3}{2}q-1}} \cos \left(\psi-\frac{3\pi}{4}\right)^q \,d\psi\sim \frac{C'_q}{n^2}$$ since the integral converges.
\end{rem}

\section{Proof of the Main Result (Theorem \protect\ref{main})}
\begin{proof}[Proof of Theorem \ref{main} assuming Lemma \ref{Lemma1}]
	First of all we remind the following
	property (see for instance \cite{M e Peccati}, page 98): let $Z_1, Z_2$ be jointly Gaussian; then, for all $q_1,
	q_2 \geq 0$ 
	\begin{equation}  \label{propertyH}
	\mathbb{E}[H_{q_1}(Z_1)H_{q_2}(Z_2)] = q_1!\delta_{q_1}^{q_2} \mathbb{E}[Z_1Z_2].
	\end{equation}
	Now we start computing the variance of $h_{\beta;2}$; hence,
	\begin{equation}\label{2}
	\begin{split}
	& \Var(h_{\beta;2})=\Var \bigg[ \int_{\mathbb{S}^2} H_2(\bar{T}_{{\alpha_{n,\beta}}}(x)) \, dx\bigg]=\mathbb{E} \bigg[ \int_{\mathbb{S}^2} H_2(\bar{T}_{{\alpha_{n,\beta}}}(x)) \, dx  \bigg]^2\\&= \mathbb{E} \bigg[ \int_{\mathbb{S}^2 \times\mathbb{S}^2  }\!\!\!\! H_2(\bar{T}_{{\alpha_{n,\beta}}}(x)) H_2(\bar{T}_{{\alpha_{n,\beta}}}(y))\, dx dy \bigg]=\int_{\mathbb{S}^2\times \mathbb{S}^2} \!\!\!\! \mathbb{E}[H_2(\bar{T}_{{\alpha_{n,\beta}}}(x)) H_2(\bar{T}_{{\alpha_{n,\beta}}}(y))] \,dx dy
	\end{split}
	\end{equation}
	which is, in view of (\ref{propertyH}), equal to
	$$\int_{\mathbb{S}^2  \times \mathbb{S}^2} 2\bar{\Gamma}_{\alpha_{n,\beta}}(x,y)^2 \,dx dy.$$ Using (\ref{cov}) we get
	\begin{equation}\label{3.3}
	\begin{split}
	\Var(h_{\beta;2})=2 C_{n,\beta}^2 \int_{\mathbb{S}^2  \times \mathbb{S}^2}\sum_{\ell={\alpha_{n,\beta}}n}^{n} \sum_{\ell'={\alpha_{n,\beta}} n}^{n}\frac{2\ell+1}{4\pi} P_\ell(\langle x,y \rangle) \frac{2\ell'+1}{4\pi} P_{\ell'}(\langle x,y \rangle) \, dx dy
	\end{split}
	\end{equation}
	and exchaging integrals and sums, and 
	applying the duplication property (see \cite{M e Peccati}, Ch. 3), that is,
	$$ \int_{ \mathbb{S}^2} \dfrac{2\ell+1}{4\pi} P_\ell(\langle x,y \rangle) \dfrac{2\ell^\prime+1}{4\pi} P_{\ell^\prime}(\langle y,z \rangle) \,dy = \dfrac{2\ell+1}{4\pi} P_\ell(\langle x,z \rangle) \delta_{\ell}^{\ell^\prime},$$
	(\ref{3.3}) is equal to
	$$ 2 C_{n,\beta}^2 \sum_{\ell} \frac{2\ell+1}{4\pi} \int_{ \mathbb{S}^2} P_\ell(\langle x,x \rangle)\,dx$$
	and since $P_\ell(0)=1 \mbox{ } \mbox{ }\forall \ell$, we conclude that
	\begin{equation*}
	\begin{split}
	\Var(h_{\beta;2})  &=2 C_{n,\beta}^2 \sum_{\ell=\alpha_{n,\beta} n}^{n} \frac{2\ell+1}{4\pi} \int_{ \mathbb{S}^2} \,dx=2 C_{n,\beta}^2 \sum_{\ell=\alpha_{n,\beta} n}^{n} \frac{2\ell+1}{4\pi} 4\pi \\& =2 C_{n,\beta}^2 [( n(n+1)-\alpha_{n,\beta} n(\alpha_{n,\beta} n-1))+n+1-\alpha_{n,\beta} n ]
	\\&= 2 C_{n,\beta}^2(n^2(1-\alpha_{n,\beta}^2)+2n+1) =\dfrac{2 (4\pi)^2}{n^2(1-\alpha_{n}^2)+2n+1}= \dfrac{2(4\pi)^2}{n^{2-\beta}+2n+1}.
	\end{split}
	\end{equation*}
	Finally, Taylor expansion implies
	\begin{equation}\label{C2}
	\begin{split}
	\Var(h_{\beta;2})=& \dfrac{32\pi^2}{n^{2-\beta}} \left(  1-\frac{2}{n^{1-\beta}}+o\left(\frac{1}{n^{1-\beta}}\right)\right)= \dfrac{32\pi^2}{n^{2-\beta}}- \dfrac{64\pi^2}{n^{3-2\beta}} +o\left(\frac{1}{n^{3-2\beta}}\right).
	\end{split}
	\end{equation}
	Let us focus now on the variance of the chaotic components $h_{\beta;q}$, for $q>2$. Hence, same computations as in (\ref{2}) lead to
	\begin{equation*}\label{eq}
	\begin{split}
	\Var(h_{\beta;q})&=\Var\bigg(\int_{ \mathbb{S}^2}H_q(\bar{T}_{\alpha_{n,\beta}}(x))\,dx\bigg)=\int_{\mathbb{S}^2 \times \mathbb{S}^2}\mathbb{E}[H_q(\bar{T}_{\alpha_{n,\beta}}(x))H_{q'}(\bar{T}_{\alpha_{n,\beta}}(y))]\, dx dy
	\\&	=q!\int_{\mathbb{S}^2  \times \mathbb{S}^2} \bar{\Gamma}_{\alpha_{n,\beta}}(x,y)^q \,dx dy
	\end{split}
	\end{equation*}
	and for the isotropy it is
	\begin{equation}\label{integ}
	\begin{split}
	&	= 2\pi |\mathbb{S}^2|q!\int_{0}^{\pi} \bar{\Gamma}_{\alpha_{n,\beta}}(\cos \theta)^q \sin \theta d\theta= 2\pi |\mathbb{S}^2|q!\int_{0}^{1/n} \bar{\Gamma}_{\alpha_{n,\beta}}(\cos \theta)^q \sin \theta d\theta\\&+2\pi |\mathbb{S}^2 |q!\int_{1/n}^{1/{(\alpha_{n,\beta} n)}} \bar{\Gamma}_{\alpha_{n,\beta}}(\cos \theta)^q \sin \theta d\theta +
	2\pi |\mathbb{S}^2 |q!\int_{1/{(\alpha_{n,\beta} n)}}^{\pi} \bar{\Gamma}_{\alpha_{n,\beta}}(\cos \theta)^q \sin \theta d\theta.
	\end{split}
	\end{equation}
	Note that for $ \theta \in [0,1/n]$, since $| \bar{\Gamma}_{\alpha_{n,\beta}}(x,y)| \leq 1$, changing variable $\theta=\dfrac{\psi}{n}$, we have that 
	\begin{equation}\label{2var}
	\int_{0}^{1/n} \bar{\Gamma}_{\alpha_{n,\beta}}(\cos \theta)^q \sin \theta \,d\theta =O \left( \int_{0}^{1/n} |\sin \theta |\, d\theta \right)= O \left( \frac{1}{n} \int_{0}^{1} \frac{\psi}{n} \,d\psi \right)= O \left(\frac{1}{n^{2}} \right).
	\end{equation}
	In the same way
	\begin{equation}\label{3var}
	\begin{split}
	\int_{1/n}^{1/(\alpha_{n,\beta} n)} &\bar{\Gamma}_{\alpha_{n,\beta}}(\cos \theta)^q \sin \theta \,d\theta =O \left( \int_{1/n}^{1/(\alpha_{n,\beta} n)} |\sin \theta |\, d\theta \right)= O \left( \frac{1}{n} \int_{1}^{1/\alpha_{n,\beta}} \frac{\psi}{n} \,d\psi \right)\\&= O \left(\frac{1}{2n^{2}} \left(\frac{1}{\alpha_{n,\beta}^2}-1\right) \right) =O \left(\frac{1}{2n^{2}} \left(1+\frac{1}{n^\beta} -1\right)\right)=O\left( \frac{1}{n^{2+\beta}} \right)
	\end{split}
	\end{equation}
	where we have used the geometric series for $\frac{1}{\alpha_{n,\beta}^2}=\frac{1}{1-1/n^\beta}$.
	For the last integral in (\ref{integ}), changing the variables as $\theta= \frac{\psi}{\alpha_n n}$, we have that
	\begin{equation}
	\begin{split}
	&\int_{1/(\alpha_{n,\beta} n)}^{\pi} \bar{\Gamma}_{\alpha_{n,\beta}}(\cos \theta)^q \sin \theta d\theta = 
	\frac{1}{\alpha_{n,\beta} n}\int_{1}^{\pi\alpha_{n,\beta} n} \bar{\Gamma}_{\alpha_{n,\beta}}\left(\cos \frac{\psi}{\alpha_{n,\beta} n}\right)^q\sin \left(\frac{\psi}{\alpha_{n,\beta} n}\right)d\psi\\&
	=\frac{1}{\alpha_{n,\beta} n}\int_{1}^{n^{\beta}} \bar{\Gamma}_1^q\sin \left(\frac{\psi}{\alpha_{n,\beta} n}\right)d\psi+\frac{1}{{\alpha_{n,\beta} n}}\int_{n^{\beta}}^{\pi\alpha_{n,\beta} n} \bar{\Gamma}_{2}^q\sin \left(\frac{\psi}{\alpha_{n,\beta} n}\right)d\psi=:I_1+I_2.
	\end{split}
	\end{equation}
	From Lemma \ref{Lemma1}, taking the $q-$power, we see that
	\begin{equation}\label{gammaq1}
	\begin{split}
	\Gamma_1^q&=\frac{C_{n,\beta}^q}{(4\pi)^q}  \frac{(\alpha_{n,\beta} n)^{2q}}{\psi^{q/2}} \frac{1}{2^q}\sqrt{\frac{2}{\pi}}^q  \frac{1}{n^{\beta q}} \bigg(-\cos \left(\psi-\frac{3\pi}{4}\right) \frac{\psi}{4n^{\beta}}- \sin  \left(\psi-\frac{3\pi}{4}\right) \bigg)^q\\&+ O \left( C_{n,\beta}^q \frac{(\alpha_{n,\beta} n)^{2q}}{n^{\beta q}} \frac{1}{\psi^{q/2+1}} \bigg(-\cos \left(\psi-\frac{3\pi}{4}\right) \frac{\psi}{4n^{2\beta}}- \sin  \left(\psi-\frac{3\pi}{4}\right)\bigg)^{q-1}\right)
	\end{split}%
	\end{equation}
	and 
	\begin{equation}\label{gammaq2}
	\begin{split}
	\Gamma_2^q&=\frac{C_{n,\beta}^q}{(4\pi)^q} 2^q \left(\frac{2}{\pi}\right)^{q/2}(\alpha_{n,\beta} n )^{2q}
	\sqrt{\psi}^q (-2)^q
	\sin \left(\frac{\psi}{2}-\frac{3\pi}{4}+\frac{(n+1)\psi}{2\alpha_{n,\beta} n}\right)^q \sin\left(\frac{(n+1)\psi}{2n\alpha_{n,\beta}}-\frac{\psi}{2}\right)^q \\& \left[ \frac{1}{\psi^2}+\frac{1}{(\alpha_{n,\beta}n)^2}\right]^q
	+O\left(C_{n,\beta}^q   \frac{(\alpha_{n,\beta} n )^{2q}}{\psi^{-q/2+2}} \sin \left(\frac{\psi}{2}-\frac{3\pi}{4}+\frac{(n+1)}{2\alpha_{n,\beta} n} \psi\right)^{q-1} \right. \\&\left.\sin\left(\frac{n+1}{2n\alpha_{n,\beta}}\psi-\frac{\psi}{2}\right)^{q-1}\times\left[ \frac{1}{\psi^2}+\frac{1}{(\alpha_{n,\beta}n)^2}\right]^{q-1} \right).
	\end{split}
	\end{equation}
	As far as $I_2$ is concerned, using the Taylor expansion of $\sin \psi/(\alpha_{n,\beta} n)$, the Newton's binomial formula and (\ref{gammaq2}) we get
	\begin{equation}
	\begin{split}
	&I_2=\frac{1}{\alpha_{n,\beta} n } \int_{n^\beta}^{n\alpha_{n,\beta}\pi} \bar{\Gamma}_2^q \sin\frac{\psi}{\alpha_{n,\beta}} \,d\psi=\frac{1}{(\alpha_{n,\beta}n)^2}\int_{n^\beta}^{n\alpha_{n,\beta}\pi}\bar{\Gamma}_2^q \psi\,d\psi\\&+O\left(\frac{1}{(\alpha_{n,\beta}n)^4}\int_{n^\beta}^{n\alpha_{n,\beta}\pi}\bar{\Gamma}_2^q \psi^3\,d\psi\right)= \left(\frac{C_{n,\beta}}{4\pi} \right)^q (2)^q \left(\frac{2}{\pi}\right)^{\frac{q}{2}} (-2)^q \frac{(\alpha_{n,\beta} n)^{2q}}{(\alpha_{n,\beta}n)^2}
	\int_{n^\beta}^{\alpha_{n,\beta} n \pi-\epsilon} \!\!\! \sqrt{\psi}^q \\&\times \sum_{k=0}^{q} \binom{q}{k} \frac{1}{\psi^{2k}} \frac{1}{(\alpha_{n,\beta}n)^{2(q-k)}} \psi \,d\psi
	\times \left(\sin \left(\frac{\psi}{2}-\frac{3\pi}{4}+\frac{(n+1)}{2\alpha_{n,\beta} n} \psi\right) \sin\left(\frac{n+1}{2n\alpha_{n,\beta}}\psi-\frac{\psi}{2}\right)\right)^q \,d\psi\\&
	+ O\left( \frac{1}{(\alpha_{n,\beta} n )^2}  \int_{n^{\beta}}^{\pi n \alpha_{n,\beta}-\epsilon} C_{n,\beta}^q   \frac{(\alpha_{n,\beta} n )^{2q}}{\psi^{-q/2+1}} \sin \left(\frac{\psi}{2}-\frac{3\pi}{4}+\frac{(n+1)}{2\alpha_{n,\beta} n} \psi\right)^{q-1}\right.\\& \times \left. \sin\left(\frac{n+1}{2n\alpha_{n,\beta}}\psi-\frac{\psi}{2}\right)^{q-1}\left[ \frac{1}{\psi^2}+\frac{1}{(\alpha_{n,\beta}n)^2}\right]^{q-1} \,d\psi \right)+ O\left( \frac{1}{(\alpha_{n,\beta}n)^{4}} \int_{n^{\beta}}^{\pi n \alpha_{n,\beta}} \Gamma_2^q \psi^{3} \,d\psi\right).
	\end{split}
	\end{equation}
	Let us focus on $ (\alpha_{n,\beta} n)^{2q-2}\frac{1}{(n^{2-\beta}+2n+1)^q}$. Taylor expansion implies that
	\begin{equation}\label{Taylorexp}
	\begin{split}
	\frac{n^{2q-2}}{(n^{2-\beta})^q} \frac{1}{\left( 1+\frac{2n+1}{n^{2-\beta}}\right) ^q} \left(1-\frac{1}{n^{\beta}}\right)^{q-1}&= \frac{n^{-2}}{n^{-\beta q}} \left(1+O\left(\frac{2n+1}{n^{2-\beta}}\right)\right) \left( 1+O\left(\frac{1}{n^{\beta}}\right)\right) \\&=\frac{n^{-2}}{n^{-\beta q}} \left(1+O\left(\frac{1}{n^{\beta}}\right)+O\left(\frac{1}{n^{1-\beta}}\right)\right).
	\end{split}
	\end{equation}
	Hence for $q \geq 4$, bounding the sine by 1 and solving the integral in $\psi$, we can see that $$I_2\ll C_q\frac{n^{q\beta}}{n^2}  \frac{1}{(\alpha_{n,\beta} n)^{2q}} (\alpha_{n,\beta}n)^{q/2+2} =O\left(\frac{1}{n^2}\right);$$
	while for $q=3$, we have to compute
	\begin{equation}\label{q3}
	\begin{split}
	C_3 \frac{n^{3\beta}}{n^2} &\int_{n^\beta}^{\alpha_{n,\beta} n \pi} \psi^{3/2+1}
	\left( \frac{1}{\psi^6}+ \frac{1}{\psi^4 (\alpha_{n,\beta}n)^2}+ \frac{1}{\psi^2 (\alpha_{n,\beta}n)^4}+\frac{1}{(\alpha_{n,\beta}n)^6} \right)\\&
	\times \sin \left(\psi-\frac{3\pi}{4}+\frac{(n+1)}{2\alpha_{n,\beta} n} \psi \right)^3 \sin\left(\frac{n+1}{2\alpha_{n,\beta}n}\psi-\psi\right)^3 \,d\psi .
	\end{split}
	\end{equation}
	In view of the following formula
	\begin{equation*}
	\begin{split}
	\sin(a)^3\sin(b)^3&= \frac{1}{32} \bigg[-3 \cos(a-3b)+\cos(3a-3b)+9\cos(a-b)-3\cos(3a-b)\\&-9 \cos (a+b)+3\cos(3a+b)+3\cos(a+3b)-\cos(3a+3b)\bigg],
	\end{split}
	\end{equation*}
	integration by parts gives the convergence of the integral in (\ref{q3}) multiplied by $n^{3\beta}$ so that (\ref{q3}) is $O(1/n^2)$.
	Regarding $I_1$ applying again the Taylor expansion of $\sin\psi/(\alpha_{n,\beta}n)$ and using (\ref{gammaq1}) we get
	\begin{equation}\label{perq3}
	\begin{split}
	&I_1=\frac{1}{\alpha_{n,\beta} n } \int_{1}^{n^\beta} \Gamma_1^q \sin\frac{\psi}{\alpha_{n,\beta}} \,d\psi=\frac{1}{(\alpha_{n,\beta}n)^2}\int_{1}^{n^\beta} \Gamma_1^q \psi\,d\psi+O\left(\frac{1}{(\alpha_{n,\beta}n)^4}\int_{1}^{n^\beta} \Gamma_1^q \psi^3\,d\psi\right)\\&
	= \frac{C_{n,\beta}^q}{(4\pi)^q} \frac{1}{2^q}\left(\frac{2}{\pi}\right)^{q/2}\!\!\! \frac{1}{n^{\beta q}}(\alpha_{n,\beta} n)^{2q-2}  \int_{1}^{n^\beta} \frac{1}{\psi^{q/2}}\psi \bigg(-\cos \left(\psi-\frac{3\pi}{4}\right) \frac{\psi}{4n^{\beta}}- \!\!\sin  \left(\psi-\frac{3\pi}{4}\right) \bigg)^q\!\! \,d\psi\\&+ O \left( C_{n,\beta}^q \frac{(\alpha_{n,\beta} n)^{2q-2}}{n^{\beta q}} \int_{1}^{n^\beta}\frac{1}{\psi^{q/2}} \bigg(-\cos \left(\psi-\frac{3\pi}{4}\right) \frac{\psi}{4n^{2\beta}}- \sin  \left(\psi-\frac{3\pi}{4}\right)\bigg)^{q-1}\,d\psi\right).
	\end{split}
	\end{equation}
	Then, bounding $\bigg|\left(-\cos\left(\psi-\frac{3\pi}{4}\right) \frac{\psi}{4n^{\beta}}- \sin  \left(\psi-\frac{3\pi}{4}\right)\right)^q\bigg|\leq |(\frac{\psi}{4n^{\beta}}+1)^q|$ and using the Newton's binomial formula, we can deduce that, for all $q\geq 4$,
	\begin{equation}
	\begin{split}
	I_1&\sim\frac{C_{n,\beta}^q}{(4\pi)^q} \frac{1}{2^q}\left(\frac{2}{\pi}\right)^{q/2} \frac{1}{n^{\beta q}}(\alpha_{n,\beta} n)^{2q-2} \int_{1}^{n^\beta} \frac{1}{\psi^{q/2-1}}
	\sum_{k=0}^{q} \binom{q}{k} \left( \frac{\psi}{4n^\beta}\right)^k  \,d\psi\\&
	= \frac{C_{n,\beta}^q}{(4\pi)^q} \frac{1}{2^q}\left(\frac{2}{\pi}\right)^{q/2}\frac{1}{n^{\beta q}}(\alpha_{n,\beta} n)^{2q-2}\sum_{k=0}^{q} \binom{q}{k} \left(\frac{1}{4n^\beta}\right)^k \int_{1}^{n^\beta} \psi^{1-q/2+k}  \,d\psi
	\end{split}
	\end{equation}
	and (\ref{Taylorexp}) gives
	\begin{equation*}
	\begin{split}
	I_1&=O\bigg( \frac{1}{n^2} \bigg(\sum_{k=0, k\ne q/2-1,k\ne q/2-2}^{q} \binom{q}{k} \left(\frac{1}{4n^\beta}\right)^k  \psi^{2 -\frac{q}{2}+k}\bigg|_{1}^{n^\beta} +\binom{q}{\frac{q}{2}-1} \left(\frac{1}{4n^\beta}\right)^{\frac{q}{2}-1}(n^\beta-1)\\&+ \binom{q}{q/2-2} \left(\frac{1}{4n^\beta}\right)^{q/2-2} \int_{1}^{n^\beta} \psi^{-1} \,d\psi\bigg) \bigg)=O\left(\frac{1}{n^2}\frac{1}{n^{q\beta/2-2\beta}}\right)+O\left(\frac{1}{n^2}\frac{\log n}{n^{q\beta/2-2\beta}}\right)
	\end{split}
	\end{equation*}
	Note that for $q=4$ we obtain $I_1=O(\log n /n^2)$.
	It remains to study the case $q=3$. Developping the third power in (\ref{perq3}), and in view of (\ref{Taylorexp}) we get
	\begin{equation*}
	\begin{split}
	&
	I_1 \sim \frac{C_q}{n^2}
	\int_{1}^{n^{\beta}} \frac{1}{\sqrt{\psi}} \bigg(\frac{\psi^{3}}{4n^{3\beta}}  \cos \left(\psi -\frac{3\pi}{4}\right)^3- \sin\left(\psi- \frac{3\pi}{4}\right)^{3}\\&- \frac{3\psi^{2}}{16n^{2\beta}}  \cos \left(\psi -\frac{3\pi}{4}\right)^2 \sin\left(\psi- \frac{3\pi}{4}\right)- \frac{3\psi}{4n^{\beta}}  \cos \left(\psi -\frac{3\pi}{4}\right)\sin\left(\psi- \frac{3\pi}{4}\right)^{2}\bigg)\,d\psi.
	\end{split}
	\end{equation*}
	Integrating by parts it can be seen that the integral converges. For example,
	\begin{equation*}
	\begin{split}
	&\int_{1}^{n^\beta} \!\!\!\frac{\sin \left(\psi - \frac{3 \pi}{4} \right)^3}{\psi^{1/2}} \,d\psi = \int_{1}^{n^{\beta}}\!\!\! \frac{\sin \left(\psi - \frac{3 \pi}{4} \right)}{\psi^{1/2}}-\frac{\cos \left(\psi - \frac{3 \pi}{4} \right)^2 \sin \left(\psi - \frac{3 \pi}{4} \right) }{\psi^{1/2}} \,d\psi = \!\!\bigg[ \frac{\cos \left(\psi - \frac{3 \pi}{4} \right)}{\psi^{1/2}}\bigg|_{1}^{n^\beta}\\&- \int_{1}^{\pi n^\beta} \frac{\cos \left(\psi - \frac{3 \pi}{4} \right)}{\psi^{3/2}} \,d\psi  \bigg] -\bigg[ \frac{\cos \left(\psi - \frac{3 \pi}{4} \right)^3}{\psi^{1/2}}\bigg|_{1}^{n^\beta}- \int_{1}^{\pi n^\beta} \frac{\cos \left(\psi - \frac{3 \pi}{4} \right)^3}{\psi^{3/2}} \,d\psi  \bigg] < \infty;
	\end{split}
	\end{equation*}
	Similarly, for the other terms we find $O(\frac{1}{n^{\beta}}n^{\beta/2})$, $O(\frac{1}{n^{2\beta}} n^{3/2\beta})$ and $O(\frac{1}{n^{3\beta}} n^{5/2\beta})$.
	Putting together the results for $I_1$ and $I_2$ with (\ref{2var}) and (\ref{3var}) the thesis of the Theorem follows.
\end{proof}

\section{Proof of Lemma \protect\ref{Lemma1}}
Before proving Lemma \ref{Lemma1}, using the same notation as in \cite{FXA},
we recall the Hilb's asymptotic formula (see \cite{szego}, Theorem 8.21.12):
\begin{equation}  \label{hilbs}
P_n^{(1,0)}(\cos \theta ) = \bigg(\sin \frac{\theta}{2}\bigg)^{-1} \bigg\{ %
\bigg(\frac{\theta}{\sin \theta}\bigg)^{1/2} J_1((n+1)\theta)
+R_{1,n}(\theta) \bigg\}
\end{equation}
where $P_n^{(1,0)}$ is a Jacobi Polynomial, which in general is defined as 
\begin{equation*}
P_n^{(\alpha, \beta)}(x)= \sum_{s=0}^n \binom{n+\alpha}{s} \binom{n+\beta}{%
	n-s} \bigg(\frac{x-1}{2}\bigg)^{n-s}\bigg(\frac{x+1}{2}\bigg)^s,
\end{equation*}
\begin{equation}\label{error}
R_{1,n}(\theta) =%
\begin{cases}
\theta^3 O(n), & 0\leq \theta \leq c/n \\ 
\theta^{1/2} O(n^{-3/2}), & c/n \leq \theta \leq \pi -\epsilon%
\end{cases}%
\end{equation}
and $J_1$ is the Bessel function of order 1.
\begin{proof}[Proof of Lemma \ref{Lemma1}]
	Looking at the covariance function in (\ref{cov}) we can write $\bar{\Gamma}_{\alpha_{n,\beta}}(x,y)$ as \begin{equation}\label{var1}
	\begin{split}
	\bar{\Gamma}_{\alpha_{n,\beta}}(x,y)&= 
	C_{n,\beta} \bigg(\sum_{\ell=0}^n \frac{2\ell+1}{4\pi} P_{\ell}(\langle x,y \rangle)-\sum_{\ell=0}^{n\alpha_{n,\beta}-1} \frac{2\ell+1}{4\pi} P_{\ell}(\langle x,y \rangle )\bigg).
	\end{split}
	\end{equation}
	Thanks to the following formula (\cite{FXA}, page 6), derived by the Christoffel-Darboux formula (\cite{AH12}), 
	\begin{equation}\label{CDformula}
	\sum_{\ell=0}^{n} \sum_{m=-\ell}^{\ell} Y_{\ell, m} (x) Y_{\ell, m} (y)= \frac{n+1}{4\pi} P_n^{(0,1)}(\cos \theta(x,y)),
	\end{equation} 
	and to the addition formula (\cite{M e Peccati} page 66):
	\begin{equation}\label{addition formula}
	\sum_{m=-\ell}^{\ell} Y_{\ell, m} (x) Y_{\ell, m} (y)= \frac{2\ell+1}{4\pi} P_\ell(\cos \theta(x,y)),
	\end{equation}
	we obtain that 
	$$\bar{\Gamma}_{\alpha_{n,\beta}}(\cos \theta)=C_{n,\beta} \bigg[ \frac{n+1}{4\pi} P_n^{(1,0)}(\cos \theta(x,y))-\frac{n\alpha_{n,\beta}}{4\pi} P_{n\alpha_{n,\beta}-1}^{(1,0)} (\cos \theta(x,y)) \bigg].$$
	Applying the Hilb's asymptotics (\ref{hilbs}) we get 
	\begin{equation}
	\begin{split}
	\bar{\Gamma}_{\alpha_{n,\beta}}(\cos \theta)=\frac{C_{n,\beta}}{4\pi}  \left(\sin \frac{\theta}{2}\right)^{-1}&\bigg[ (n+1) \left(\frac{\theta}{\sin \theta}\right)^{1/2} J_1((n+1)\theta)+(n+1)R_{1,n}(\theta)+\\&-n{\alpha_{n,\beta}} \left(\frac{\theta}{\sin \theta}\right)^{1/2} J_1(n \alpha_{n,\beta} \theta)- n\alpha_{n,\beta} R_{1,n\alpha_{n,\beta}}(\theta)  \bigg].
	\end{split}
	\end{equation}
	In view of (\ref{error}) the error term $(n+1)R_{1,n}(\theta)- n \alpha_{n,\beta} R_{1,n\alpha_{n,\beta}}(\theta) $,
	for $\dfrac{c}{n\alpha_{n,\beta}} \leq \theta \leq \pi-\epsilon$, it is equal to
	\begin{equation*}
	\theta^{\frac{1}{2}} O\!\left(n^{-\frac{3}{2}}\right)(n+1)-n \alpha_{n,\beta} \theta^{\frac{1}{2}} O\!\left(n^{-\frac{3}{2}} \alpha_{n,\beta}^{-3/2}\right)=\theta^{\frac{1}{2}} O\!\left(n^{-\frac{1}{2}}-n^{-\frac{1}{2}} \alpha_{n,\beta}^{-\frac{1}{2}}\right)= O\!\left(\frac{\theta^{\frac{1}{2}}}{\sqrt{n}}\right)=O\left(\frac{1}{\sqrt{n}}\right).
	\end{equation*}
	Now, changing variable $\theta= \frac{\psi}{\alpha_{n,\beta} n}$ and exploiting the expansion of the Bessel functions (\cite{szego}, page 15-16),
	\begin{equation}\label{J1}
	J_1 (x) = \bigg( \frac{2}{\pi x} \bigg)^{1/2} \cos \left(x - \frac{3\pi }{4}\right)- \dfrac{3}{4\sqrt{2\pi}x^{3/2}} \sin\left(x - \frac{3\pi }{4}\right)+O(x^{-5/2}), \mbox{ as } x \rightarrow \infty,
	\end{equation}
	we find
	\begin{equation*}
	\begin{split}
	&\bar{\Gamma}_{\alpha_{n}}\left(\cos \frac{\psi}{\alpha_{n,\beta} n}\right)=\frac{C_{n,\beta}}{4\pi}  \bigg(\sin \frac{\psi}{2\alpha_{n,\beta} n}\bigg)^{-1} \bigg\{ \left(\frac{\psi/(n\alpha_{n,\beta})}{\sin \psi/(\alpha_{n,\beta} n)}\right)^{1/2} \bigg[(n+1) \bigg( \bigg( \frac{2\alpha_{n,\beta} n}{\pi (n+1)\psi} \bigg)^{1/2} \\& \times\cos\! \left(\!\frac{\psi(n+1)}{\alpha_{n,\beta} n}-\!\!\frac{3\pi}{4}\right)-\frac{3}{4\sqrt{2\pi}} \left(\frac{1}{\psi}\right)^{3/2}\!\! \left(\frac{\alpha_{n,\beta} n}{n+1}\right)^{3/2}\! \!\!\sin \left(\frac{n+1}{\alpha_{n,\beta} n } \psi -\frac{3\pi}{4} \right)\!\!+\! O\left(\frac{\alpha_{n,\beta} n}{(n+1)\psi}\right)^{\frac{5}{2}} \!\!\bigg)\\&
	-n\alpha_{n,\beta} \bigg(\sqrt{\frac{2}{\pi \psi}} \cos\left(\psi-\frac{3\pi}{4}\right)- \frac{3}{4\sqrt{2\pi}} \left(\frac{1}{\psi}\right)^{3/2} \sin \left( \psi- \frac{3\pi}{4} \right) +O(1/\psi^{5/2}) \bigg) \bigg]+ O\left(\frac{1}{\sqrt{n}}\right)\bigg\}
	\end{split}
	\end{equation*}
	which is equal to
	\begin{equation*}
	\begin{split}
	&\bar{\Gamma}_{\alpha_{n}}\left(\cos \frac{\psi}{\alpha_{n,\beta} n}\right)
	=\frac{C_{n,\beta}}{4\pi}  \bigg(\sin \frac{\psi}{2\alpha_{n,\beta} n}\bigg)^{-1} \bigg\{ \left(\frac{\psi/(n\alpha_{n,\beta})}{\sin \psi/(\alpha_{n,\beta} n)}\right)^{1/2} \alpha_{n,\beta} n 
	\bigg[ 
	\sqrt{\frac{2}{\pi \psi}} \bigg( \bigg(\frac{n+1}{\alpha_{n,\beta} n} \bigg)^{1/2} \\& \times\cos \left((n+1)\frac{\psi}{\alpha_{n,\beta} n}-\frac{3\pi}{4}\right) - \cos \left( \psi- \frac{3\pi}{4} \right) \bigg)\\&
	+\frac{3}{4\sqrt{2\pi}} \left(\frac{1}{\psi}\right)^{3/2} \left( \sin\left(\psi-\frac{3\pi}{4}\right)-\left( \frac{\alpha_{n,\beta} n}{n+1}\right)^{1/2} \sin \left(\frac{n+1}{\alpha_{n,\beta} n } \psi -\frac{3\pi}{4} \right)\right)+ O\left(\frac{1}{\psi^{5/2}n^\beta} \right) \bigg]\bigg\}.
	\end{split}
	\end{equation*}
	Using the Taylor expansion 
	\begin{equation}\label{taylor}
	(1+x)^{\gamma}= 1+\gamma x+\frac{\gamma (\gamma-1)}{2}x^2+o(x^2) \mbox{ }\mbox{ }\mbox{ }\mbox{ }\mbox{ } (x \rightarrow 0)
	\end{equation} 
	applied to
	\begin{align*}
	\left(\frac{n+1}{\alpha_{n,\beta} n}\right)^{1/2}\!\!&=&\!\!\left(1+\frac{1}{n}\right)^{1/2}\left(1-\frac{1}{n^{\beta}}\right)^{-1/4}= 1+\frac{1}{4n^{\beta}}+\frac{5}{32n^{2\beta}}+\frac{1}{2n}+o\left(\frac{1}{n^{2\beta}}\right)+o\left(\frac{1}{n}\right),\\
	\left(\frac{\alpha_{n,\beta} n}{n+1}\right)^{1/2}\!\!&=&\!\!\left(1+\frac{1}{n}\right)^{-1/2}\left(1-\frac{1}{n^{\beta}}\right)^{1/4}= 1-\frac{1}{4n^{\beta}}-\frac{3}{32n^{2\beta}}-\frac{1}{2n}+o\left(\frac{1}{n^{2\beta}}\right)+o\left(\frac{1}{n}\right)
	\end{align*}
	we obtain
	\begin{equation}\label{uu}
	\begin{split}
	\bar{\Gamma}_{\alpha_{n}}&\left(\cos \frac{\psi}{\alpha_{n,\beta} n}\right)=\frac{C_{n,\beta}}{4\pi}  \bigg(\sin \frac{\psi}{2\alpha_{n,\beta} n}\bigg)^{-1} \bigg\{ \left(\frac{\psi/(n\alpha_{n,\beta})}{\sin \psi/(\alpha_{n,\beta} n)}\right)^{1/2} \alpha_{n,\beta} n 
	\bigg[ 
	\sqrt{\frac{2}{\pi \psi}} \\& \times\left( \cos \left(\frac{\psi(n+1)}{\alpha_{n,\beta} n}-\frac{3\pi}{4}\right) -  \cos \left( \psi- \frac{3\pi}{4} \right) \right) +\sqrt{\frac{2}{\pi \psi}} \frac{1}{4n^{\beta}} \cos \left((n+1)\frac{\psi}{\alpha_{n,\beta} n}-\frac{3\pi}{4}\right) \\&
	+\frac{3}{4\sqrt{2\pi}} \left(\frac{1}{\psi}\right)^{3/2} \left( \sin\left(\psi-\frac{3\pi}{4}\right)- \sin \left(\frac{n+1}{\alpha_{n,\beta} n } \psi -\frac{3\pi}{4} \right) \right)\\&+ \frac{3}{4\sqrt{2\pi}} \left(\frac{1}{\psi}\right)^{3/2} \frac{1}{4n^{\beta}}\sin\left(\psi\frac{n+1}{n\alpha_{n,\beta}}-\frac{3\pi}{4}\right)+O\left(\frac{1}{\sqrt{\psi}n^{2\beta}}+ \frac{1}{\psi^{3/2}n^{2\beta}}+\frac{1}{\psi^{5/2}n^\beta} \right) \bigg]\bigg\}.
	\end{split}
	\end{equation}
	Now for $n^{\beta} <\psi<\alpha_{n,\beta}n\pi$, using that
	\begin{eqnarray*}
		\cos \!\left(\frac{\psi(n+1)}{\alpha_{n,\beta} n}-\frac{3\pi}{4} \right) -  \cos \left( \psi- \frac{3\pi}{4} \right)\!\!&=&\!\!-2 \sin \left(	\frac{\psi}{2}-\frac{3\pi}{4}+\frac{(n+1)}{2\alpha_{n,\beta} n} \psi\right) \sin \left(\frac{n+1}{n\alpha_{n,\beta}}\psi-\psi\right), \\
		\sin \!\left( \psi- \frac{3\pi}{4} \right)-\sin \left(\frac{\psi(n+1)}{\alpha_{n,\beta} n}-\frac{3\pi}{4} \right)\!\!&=&\!\!2 \cos \left(	\frac{\psi}{2}-\frac{3\pi}{4}+\frac{(n+1)}{2\alpha_{n,\beta} n} \psi\right) \sin\left(\frac{\psi}{2}-\frac{n+1}{2n\alpha_{n,\beta}}\psi\right)
	\end{eqnarray*} we find
	\begin{equation}\label{gamma2}
	\begin{split}
	\bar{\Gamma}_{2}&=
	\frac{C_{n,\beta}}{4\pi}  \bigg(\sin \frac{\psi}{2\alpha_{n,\beta} n}\bigg)^{-1} \left(\frac{\psi/(n\alpha_{n,\beta})}{\sin \psi/(\alpha_{n,\beta} n)}\right)^{1/2} \frac{\alpha_{n,\beta} n}{\sqrt{\psi}} \\&\bigg[\sqrt{\frac{2}{\pi}} (-2) \sin \left(	\frac{\psi}{2}-\frac{3\pi}{4}+\frac{(n+1)}{2\alpha_{n,\beta} n} \psi\right) \sin\left(\frac{n+1}{n\alpha_{n,\beta}}\psi-\psi\right)+O\left( \frac{1}{ \psi} \right)\bigg].
	\end{split}%
	\end{equation}
	Whereas for $\psi \in [1,n^{\beta}] $ applying Taylor expansion to $\frac{n+1}{\alpha_{n,\beta}n}$, exploiting the formulas of sum and difference of angles and calling  $\zeta= \frac{\psi}{2n^{\beta}}+\frac{3\psi}{8n^{2\beta}} +\frac{\psi}{n}+o(\frac{1}{n})+o\left(\frac{1}{n^{2\beta}}\right)$ we have that 
	\begin{eqnarray*}
		\cos \left((n+1)\frac{\psi}{\alpha_{n,\beta} n}-\frac{3\pi}{4}\right)\!\!\!&=&\!\!\! \cos \left( \psi-\frac{3\pi}{4}+\zeta\right)=\cos \left(\psi-\frac{3\pi}{4}\right) \cos \zeta- \sin \left(\psi-\frac{3\pi}{4}\right) \sin \zeta,
		\\
		\sin \left((n+1)\frac{\psi}{\alpha_{n,\beta} n}-\frac{3\pi}{4}\right)\!\!\!&=&\!\!\! \sin \left( \psi-\frac{3\pi}{4}+\zeta\right)=\sin \left(\psi-\frac{3\pi}{4}\right) \cos \zeta+ \cos \left(\psi-\frac{3\pi}{4}\right) \sin \zeta.
	\end{eqnarray*} 
	Plugging these formulas in (\ref{uu}), 
	%
	using the fact that $\cos \zeta=1-\frac{\zeta^2}{2}+O(\zeta^4)$ and  $\sin \zeta=\zeta+ O(\zeta^3)$, as $\zeta \to 0$,
	we get 
	\begin{equation}\label{gamma1}
	\begin{split}
	\bar{\Gamma}_{1}&=
	\frac{C_{n,\beta}}{4\pi}  \bigg(\sin \frac{\psi}{2\alpha_{n,\beta} n}\bigg)^{-1} \bigg\{ \left(\frac{\psi/(n\alpha_{n,\beta})}{\sin \psi/(\alpha_{n,\beta} n)}\right)^{1/2} \alpha_{n,\beta} n 
	\bigg[ \sqrt{\frac{2}{\pi\psi}} \cos \left(\psi-\frac{3\pi}{4}\right)\\& (-1)\frac{\psi^2}{8n^{2\beta}}+O\left(\frac{1}{\sqrt{\psi}} \frac{\psi^2}{n^{3\beta}}\right)- \sqrt{\frac{2}{\pi\psi}} \zeta	\sin  \left(\psi-\frac{3\pi}{4}\right) +\sqrt{\frac{2}{\pi\psi}} \frac{1}{4n^{\beta}} \cos \left(\psi-\frac{3\pi}{4}\right)  \\&-\sqrt{\frac{2}{\pi}}\frac{1}{4n^\beta} \zeta \sin\left(\psi-\frac{3\pi}{4}\right)
	+\frac{3}{4 \sqrt{2\pi}} \frac{1}{\psi^{3/2}} 
	\bigg(\sin\left(\psi-\frac{3\pi}{4}\right)\frac{\psi^2}{8n^{2\beta}}+O\left(\frac{\psi^2}{n^{3\beta}}\right)  \\&-\zeta \cos \left(\psi-\frac{3\pi}{4}\right) +\frac{1}{4n^\beta}\sin \left(\psi-\frac{3\pi}{4}\right) \bigg)+O\left(\frac{\sqrt{\psi}}{n^{2\beta}}+ \frac{1}{\psi^{3/2} n^\beta}\right)+O\left(\frac{\psi^{3/2}}{n^{3\beta}}\right) \bigg]\bigg\}
	\\&=
	\frac{C_{n,\beta}}{4\pi}  \bigg(\sin \frac{\psi}{2\alpha_{n,\beta} n}\bigg)^{-1} \left(\frac{\psi/(n\alpha_{n,\beta})}{\sin \psi/(\alpha_{n,\beta} n)}\right)^{1/2} \alpha_{n,\beta} n \frac{\sqrt{\psi}}{n^\beta}\\& \times \sqrt{\frac{2}{\pi}} \frac{1}{2} \bigg(-\cos \left(\psi-\frac{3\pi}{4}\right) \frac{\psi}{4n^{\beta}}- \sin  \left(\psi-\frac{3\pi}{4}\right) +O\left( \frac{1}{ \psi} \right)\bigg) .
	\end{split}%
	\end{equation}
	Taylor expansion of $1/\sin x $ and $ \sin x$, as $x \to 0,$ implies that
	$$\frac{1}{\sin \psi/(n\alpha_{n,\beta})} \left( \frac{\psi/(\alpha_{n,\beta}n)}{\sin \psi/(n\alpha_{n,\beta})}\right)^{1/2}=\frac{\alpha_{n,\beta} n}{\psi}+  \frac{\psi}{\alpha_{n,\beta}6}+O\left(\frac{\psi^3}{(\alpha_{n,\beta} n )^3}\right)$$
	which replaced in (\ref{gamma1}) and (\ref{gamma2}) gives 
	the thesis of the Lemma.
\end{proof}

\subsection*{Acknowledgements}
This topic was suggested by Domenico Marinucci and Igor Wigman. The author
would like to thank them also for some useful discussions. Thanks also to an anonimous referee for some comments. The author has been supported by the German Research Foundation (DFG) via RTG 2131 and by GNAMPA-INdAM (project: \emph{Stime asintotiche: principi di invarianza e grandi deviazioni}).







\end{document}